\documentclass[12pt,reqno]{amsart}

\usepackage{amssymb}
\usepackage{amsmath}
\usepackage{amscd}
\usepackage{amstext}
\usepackage{amsfonts}
\usepackage[all]{xy}

\newtheorem{thm}{Theorem}[section] 
\newtheorem{lemma}[thm]{Lemma}
\newtheorem{prop}[thm]{Proposition}
\newtheorem{cor}[thm]{Corollary}
\newtheorem{rem}[thm]{Remark}
\theoremstyle{definition}

\newcommand{\C}{\mathbb{C}}
\newcommand{\R}{\mathbb{R}}

\newcommand{\Q}{\mathbb{Q}}
\newcommand{\Z}{\mathbb{Z}}
\newcommand{\g}{\frak{g}}

\numberwithin{equation}{section}

\begin{document} 

\title[Dolbeault cohomology of Oeljeklaus-Toma manifolds]{Remarks on Dolbeault cohomology of Oeljeklaus-Toma manifolds and Hodge theory}

\author[H. Kasuya]{Hisashi Kasuya}

\address{Department of Mathematics, Graduate School of Science, Osaka University, Osaka,
Japan}

\email{kasuya@math.sci.osaka-u.ac.jp}



\begin{abstract}
We give explicit harmonic representatives of Dolbeault cohomology of Oeljeklaus-Toma manifolds and show that they are geometrically Dolbeault formal.
We also give explicit harmonic representatives of Bott-Chern cohomology of Oeljeklaus-Toma manifolds of type $(s,1)$ and study the Angella-Tomassini inequality.

\end{abstract}

\maketitle
\section{Introduction}

Let $M$ be a compact complex manifold, $A^{\ast}(M)$ the real de Rham complex of $M$ with the  exterior differential $d$ and $H^{\ast}(M)$ the de Rham cohomology.
On the  de Rham complex $A^{\ast}(M)_{\C}$, we consider  the double complex structure $(A^{\ast}(M)_{\C}=A^{\ast,\ast}(M), d=\partial+\bar\partial)$ such that  $\bar\partial $ is the Dolbeault operator.
Denote by $H^{\ast,\ast}_{\bar\partial}(M)$ the Dolbeault cohomology of $M$.
Associated with this double complex, we have the spectral sequence $E^{\ast,\ast}_{r}$ so that $E_{1}^{p,q}=H^{p,q}_{\bar\partial}(M)$ and it converges to the de Rham cohomology.
We call it the Fr\"olicher spectral sequence of $M$.
We have the inequality $\sum_{p+q=k}\dim H^{p,q}_{\bar\partial}(M)\ge \dim H^{k}(M)$ and
the equality holds  if and only if the Fr\"olicher spectral sequence of $M$ degenerates at $E_{1}$-term.

The Bott-Chern cohomology of a complex manifold $M$ is defined by $H^{\ast,\ast}_{BC}(M)=\frac{\rm ker\partial \cap ker \bar\partial}{\rm im \partial\bar\partial}$.
On a complex manifold $M$, we say that $\partial\bar\partial$-Lemma holds if the natural map $H^{\ast,\ast}_{BC}(M)\to H^{\ast}(M)_{\C}$ is injective.
In \cite{DGMS}, it is shown that   the $\partial\bar\partial$-Lemma holds if and only if
\begin{enumerate}
\item The Fr\"ohlicher spectral sequence degenerates at $E_{1} $ and
\item The filtration on $H^{k}(M)_{\C}$ induced by the filtration $F^{r}(A^{\ast}(M)_{\C})=\bigoplus_{p\ge r}A^{p,q}(M)$ is a Hodge structure of weight $k$ for every $k\ge 0$.
\end{enumerate}
It is known that on a compact K\"ahler manifold,   the $\partial\bar\partial$-Lemma holds and 
this implies that every compact K\"ahler manifold is formal   in the sense  of Sullivan (\cite{DGMS}).
We notice that on a complex manifold $M$, the $\partial\bar\partial$-Lemma also implies the Dolbeault formality as in \cite{NT}.

We are interested in these Hodge theoretical  properties on non-K\"ahler manifolds.
In \cite{OT},  Oeljeklaus and Toma introduce compact complex manifolds (OT-manifolds) associated with algebraic number fields.
OT manifolds are higher dimensional analogues of Inoue surfaces of type $S^{0}$.
In \cite{OTH}, Otiman and  Toma proves that  the Fr\"olicher spectral sequence of every OT-manifold degenerates at $E_{1}$-term.
The purpose of this paper is to study the  Dolbeault cohomology and Hodge theory on OT-manifolds in more detail by using this result.
Presenting OT-manifolds as solvmanifolds as in \cite{KV}, we give explicit harmonic representatives of Dolbeault cohomology of OT-manifolds.
By this,   we prove geometrical Dolbeault formality (and hence Dolbeault formality) of OT-manifolds.
We know that the $\partial\bar\partial$-Lemma does not  hold on any OT-manifold, since the Hodge symmetry does not hold (see \cite[Proposition 2.3, 2.5]{OT}).
Hence we obtain many examples of compact complex manifolds without $\partial\bar\partial$-Lemma  satisfying the $E_{1}$-degeneration of the Fr\"olicher spectral sequence and the Dolbeault formality.

We also give explicit harmonic representatives of Bott-Chern cohomology of OT-manifolds of type $(s,1)$.
By this, we study the  Fr\"olicher-type inequality in \cite{AT}  for Bott-Chern cohomology of OT-manifolds of type $(s,1)$.
This says  that OT-manifolds of type $(s,1)$ are "near" to the $\partial\bar\partial$-Lemma.

\section{de Rham and Dolbeault cohomology of certain solvmanifolds}
We consider the semi-direct product $G=\R^{n}\ltimes_{\phi} (\R^{n}\oplus \C^{m})$ of real abelian Lie groups $\R^{n}$ and $\R^{n}\oplus \C^{m}$ given by the homomorphism $\phi:\R^{n}\to {\rm Aut}(\R^{n}\oplus \C^{m})$ so that 
\[\phi(x)(y, z)=(e^{x_{1}}y_{1},\dots ,e^{x_{n}}y_{n}, e^{\psi_{1}(x)}z_{1},\dots,  e^{\psi_{m}(x)}z_{m})
\]
for $x=(x_{1},\dots ,x_{n})\in \R^{n}$, $(y, z)=(y_{1},\dots ,y_{n}, z_{1},\dots,  z_{m})\in \R^{n}\oplus \C^{m}$ and some linear 
functions $\psi_{1},\dots, \psi_{m}: \R^{n}\to  \C$.
$G$ is a simply connected solvable Lie group.
Suppose we have lattices $\Lambda\subset \R^{n}$ and $\Delta \subset \R^{n}\oplus \C^{m}$ so that for every $\lambda\in \Lambda $ the automorphism  $\phi(\lambda)$ on $\R^{n}\oplus \C^{m}$ preserves $\Delta$.
Then the subgroup $\Gamma=\Lambda\ltimes_{\phi} \Delta\subset G$ is a cocompact discrete subgroup of $G$.
Consider the solvmanifold $\Gamma\backslash G$.
We identify the de Rham complex $A^{\ast}(\Gamma\backslash G)$ of $\Gamma\backslash G$ with the subcomplex of  the de Rham complex $A^{\ast}( G)$ of  $G$ consisting of the $\Gamma$-invariant differential forms.

Let $\g$ be  the Lie algebra of $G$. We identify the Lie algebra complex $\bigwedge \g^{\ast}$ with the subcomplex of $A^{\ast}(\Gamma\backslash G)$ consisting of the left-$G$-invariant  differential forms.
We have a basis 
\[ dx_{1},\dots, dx_{n},e^{-x_{1}}dy_{1},\dots, e^{-x_{n}}dy_{n}, e^{-\psi_{1}(x)}dz_{1},\dots,  e^{-\psi_{m}(x)}dz_{m}, e^{-\bar\psi_{1}(x)}d\bar{z}_{1},\dots , e^{-\bar\psi_{m}(x)}d\bar{z}_{m}\]
 of $\g^{\ast}_{\C}=\g^{\ast}\otimes \C$.
We define the finite dimensional graded subspace 
\[A^{\ast} _{\Lambda}= \left \langle dx_{I}\wedge dy_{J}\wedge dz_{K}\wedge d\bar{z}_{L}\left \vert\begin{array}{cc} I, J\subset [n], K,L\subset [m],  \\ \exp\left(\sum_{j\in J}x_{j}+\sum_{k\in K}\psi_{k}(x)+\sum_{l\in L}\bar{\psi}_{l}(x)\right)\vert _{x\in \Lambda}=1  \end{array} \right.   \right\rangle.
\]
of  $A^{\ast}(\Gamma\backslash G)$ where $[n]=\{1,\dots, n\}$ and for a multi-index $I=\{i_{1},\dots i_{a}\}\subset [n]$ we write $dx_{I}=dx_{i_{1}}\wedge\dots \wedge dx_{i_{a}}$.
Obviously, the exterior differential on $A^{\ast} _{\Lambda}$ is trivial.
\begin{thm}[\cite{KM,KR,KN}]
$A^{\ast} _{\Lambda}$ is a differential graded subalgebra of $A^{\ast}(\Gamma\backslash G)$ and the inclusion $A^{\ast} _{\Lambda}\subset A^{\ast}(\Gamma\backslash G)$ induces a cohomology isomorphism.

\end{thm}

Regarding $1$-forms 
\[\alpha_{1}= dx_{1}+\sqrt{-1}e^{-x_{1}}dy_{1},\dots, \alpha_{n}=dx_{n}+\sqrt{-1}e^{-x_{n}}dy_{n}, \beta_{1}= e^{-\psi_{1}(x)}dz_{1},\dots ,\beta_{m}= e^{-\psi_{m}(x)}dz_{m}\]
as $(1,0)$-forms on $\Gamma\backslash G$, 
we have a left-$G$-invariant almost complex structure $J$ on $\Gamma\backslash G$.
We can easily check that $J$ is integrable.
\begin{rem}
We also consider $J$ as a  left-$G$-invariant  complex structure $J$ on the Lie group $G$.
The complex manifold $(G,J)$ is a universal covering of $\Gamma\backslash G$.
We have the biholomorphic map
\[G=\R^{n}\ltimes_{\phi} (\R^{n}\oplus \C^{m})\ni (x,y, z)\mapsto (y_{1}+\sqrt{-1}e^{x_{1}},\dots, y_{n}+\sqrt{-1}e^{x_{n}},z)\in H^{n}\times \C^{m}
\]
where $H$ is the complex  upper half-plane $H=\{z\in \C: {\rm Im} z>0\}=\R\times \R_{>0}$.
\end{rem}
We consider the Dolbeault complex $(A^{\ast,\ast} (\Gamma\backslash G), \bar\partial)$ of the complex manifold $(\Gamma\backslash G, J)$.
We have 
\[\bar\partial\alpha_{i}=-\frac{1}{2} \bar\alpha_{i}\wedge \alpha_{i},\qquad \bar\partial\bar{\alpha}_{i}=0,\qquad \bar\partial\beta_{i}=-\frac{1}{2}\psi_{i}(\bar\alpha)\wedge \beta_{i}\,\, {\rm and}\, \, \bar\partial\bar{\beta}_{i}=-\frac{1}{2}\bar{\psi}_{i}(\bar\alpha)\wedge \bar{\beta}_{i}
\]
where $\psi_{i}(\bar\alpha)$ and $\bar{\psi}_{i}(\bar\alpha)$ are $(0,1)$-forms associated with linear functions $\psi_{i}(x)$ and $\bar{\psi}_{i}(x)$ by putting $x=\bar\alpha=(\bar\alpha_{1},\dots, \bar\alpha_{n})$.
We define the 
 finite dimensional bigraded subspace 
\[B^{\ast,\ast} _{\Lambda}= \left \langle \exp(\Psi_{JKL}(x)) \bar{\alpha}_{I}\wedge \alpha_{J}\wedge \beta_{K}\wedge \bar{\beta}_{L}\left \vert\begin{array}{cc} I, J\subset [n], K,L\subset [m],  \\  \exp(\Psi_{JKL}(x))\vert _{x\in \Lambda}=1  \end{array} \right.   \right\rangle.
\]
of $A^{\ast,\ast} (\Gamma\backslash G)_{\C}$
where $\Psi_{JKL}(x)=\sum_{j\in J}x_{j}+\sum_{k\in K}\psi_{k}(x)+\sum_{l\in L}\bar{\psi}_{l}(x)$.
Since we have
\[\bar\partial\left( \bar{\alpha}_{I}\wedge \alpha_{J}\wedge \beta_{K}\wedge \bar{\beta}_{L}\right)=-\frac{1}{2}\Psi_{JKL}(\bar\alpha)\wedge \bar{\alpha}_{I}\wedge \alpha_{J}\wedge \beta_{K}\wedge \bar{\beta}_{L},
\]
the Dolbeault operator $\bar\partial$ on $B^{\ast,\ast} _{\Lambda}$ is trivial.
We can easily check that $B^{\ast,\ast} _{\Lambda}$ is differential bigraded  subalgebra of the Dolbeault complex $A^{\ast,\ast} (\Gamma\backslash G)$.

\begin{lemma}\label{injj}
Define the left-$G$-invariant  Hermitian metric 
\[h_{G}=\alpha_{1}\cdot \bar\alpha_{1}+\dots +\alpha_{n}\cdot \bar\alpha_{n}+\beta_{1}\cdot \bar\beta_{1}+\dots +\beta_{m}\cdot \bar\beta_{m}.
\]
Then $B^{\ast,\ast} _{\Lambda}$ consists of $\bar\partial$-harmonic forms on the Hermitian manifold $(\Gamma\backslash G,h_{G})$.
In particular, the inclusion  $B^{\ast,\ast} _{\Lambda}\subset A^{\ast,\ast} (\Gamma\backslash G)$ induces an injection $B^{\ast,\ast} _{\Lambda}\hookrightarrow H^{\ast,\ast}_{\bar\partial} (\Gamma\backslash G)$.
\end{lemma}
\begin{proof}
We consider the conjugate Hodge star operator  $\bar\ast$ associated with this metric.
Then we have
\[\bar\ast \left( \exp(\Psi_{JKL}(x)) \bar{\alpha}_{I}\wedge \alpha_{J}\wedge \beta_{K}\wedge \bar{\beta}_{L}\right)=\exp(\bar{\Psi}_{JKL}(x))\bar{\alpha}_{\check{I}}\wedge \alpha_{\check{J}}\wedge \beta_{\check{K}}\wedge \bar{\beta}_{\check{L}}
\]
where for a multi-index $I\subset [n]$ we write $\check{I}=[n]-I$.
Since $G$ admits a lattice $\Gamma$, $G$ is unimodular (see \cite[Remark 1.9]{R}).
This implies 
\[\exp({\Psi}_{[n][m][m]}(x))=\exp\left(\sum_{j\in [n]}x_{j}+\sum_{k\in [m]}\psi_{k}(x)+\sum_{l\in [m]}\bar{\psi}_{l}(x)\right)=1.\]
Thus 
\[\exp(\Psi_{\check{J}\check{K}\check{L}}(x))= \exp(-\Psi_{JKL}(x)).
\]
Since $\Delta$ is a lattice in $\R^{n}$,   if  $\exp(\Psi_{JKL}(x))=1$ for any $x\in \Lambda$, then $\exp(\Psi_{JKL}(x))$ must be unitary for any $x\in \R^{n}$.
Thus $ \exp(-\Psi_{JKL}(x))=\exp(\bar{\Psi}_{JKL}(x))$ and so we have $\exp(\Psi_{\check{J}\check{K}\check{L}}(x))=\exp(\bar{\Psi}_{JKL}(x))$.
This implies 
\[\bar\ast \left(B^{\ast,\ast} _{\Lambda}\right)\subset B^{\ast,\ast} _{\Lambda}
\]
and so $B^{\ast,\ast} _{\Lambda}$ consists of $\bar\partial$-harmonic forms on the Hermitian manifold $\Gamma\backslash G$.
Hence the  lemma follows.
\end{proof}

Obviously we have a graded algebra isomorphism ${\rm Tot}^{\ast} B^{\ast,\ast} _{\Lambda}\cong A^{\ast} _{\Lambda}$.
Thus we have the following statement.
\begin{cor}\label{iso}
If the Fr\"olicher spectral sequence of the complex manifold $(\Gamma\backslash G, J)$ degenerates at $E_{1}$-term, then 
the inclusion  $B^{\ast,\ast} _{\Lambda}\subset A^{\ast,\ast} (\Gamma\backslash G)$ induces an isomorphism $B^{\ast,\ast} _{\Lambda}\cong H^{\ast,\ast}_{\bar\partial} (\Gamma\backslash G)$.
\end{cor}

A complex manifold $M$  is Dolbeault formal if there exists a sequence of differential bigraded  algebra homomorphisms
\[H^{*,*}_{\bar\partial}(M)\leftarrow C^{*,*}_{1}\rightarrow C^{*,*}_{2}\leftarrow\cdot \cdot \cdot \leftarrow C^{*,*}_{n}\rightarrow \Omega^{*,*}(M)
\]
such that each  morphism  induces a cohomology isomorphism where $H^{*,*}_{\bar\partial}(M)$ is regarded as a differential bigraded  algebra with the trivial differential.

By the Proof of Lemma \ref{injj}, bigraded algebra  $B^{\ast,\ast} _{\Lambda}$ consists of $\bar\partial$-harmonic forms associated with  a left-$G$-invariant  Hermitian metric.
Hence we have the following.
\begin{cor}\label{for}
If the Fr\"olicher spectral sequence of the complex manifold $(\Gamma\backslash G, J)$ degenerates at $E_{1}$-term, then 
the complex manifold $(\Gamma\backslash G, J)$ is geometrically Dolbeault formal i.e. it admits Hermitian metric so that every product of $\bar\partial$-harmonic forms  is $\bar\partial$-harmonic,
in particular $(\Gamma\backslash G, J)$ is Dolbeault formal.

\end{cor}

\begin{rem}
We notice that $B^{\ast,\ast} _{\Lambda}$  contains   differential forms which are not left-$G$-invariant in general.
We can directly compute left-invariant harmonic $\bar\partial$-forms by using  the Lie algebra complex $\bigwedge \g^{\ast}_{\C}$.
But it may be difficult to compute harmonic $\bar\partial$-forms which are not left-$G$-invariant.
\end{rem}

\begin{rem}\label{sod}
Each $\Gamma\backslash G$ is real smooth $T^{n}$-bundle over $T^{n+2m}$.
This fiber bundle structure  is not holomorphic.
A standard strategy for computing Dolbeault cohomology of complex solvmanifolds is using the spectral sequence of holomorphic fiber bundles with Dolbeault-computable fibers and bases (see \cite{CFK}).
However, we can not have such fiber bundle structure on our solvmanifold $\Gamma\backslash G$ as above.

\end{rem}

\section{Oeljeklaus-Toma manifolds}
Let $K$ be a finite extension field of $\Q$ of degree $s+2t$ ($s>0$, $t>0$).
Suppose $K$ admits embeddings $\sigma_{1},\dots \sigma_{s},\sigma_{s+1},\dots, \sigma_{s+2t}$ into $\C$ such that $\sigma_{1},\dots ,\sigma_{s}$ are real embeddings and $\sigma_{s+1},\dots, \sigma_{s+2t}$ are complex ones satisfying $\sigma_{s+i}=\bar \sigma_{s+i+t}$ for $1\le i\le t$. 
For any $s$ and $t$, we can choose $K$ admitting such embeddings (see \cite{OT}).
Let ${\mathcal O}_{K}$ be the ring of algebraic integers of $K$, ${\mathcal O}_{K}^{\ast}$ the group of units in ${\mathcal O}_{K}$ and 
\[{\mathcal O}_{K}^{\ast\, +}=\{a\in {\mathcal O}_{K}^{\ast}: \sigma_{i}(a)>0 \,\, {\rm for \,\,  all}\,\, 1\le i\le s\}.
\]  
Define $\sigma :{\mathcal O}_{K}\to \R^{s}\times \C^{t}$ by
\[\sigma(a)=(\sigma_{1}(a),\dots ,\sigma_{s}(a),\sigma_{s+1}(a),\dots ,\sigma_{s+t}(a))
\]
for $a\in {\mathcal O}_{K}$.
Define $l:{\mathcal O}_{K}^{\ast\, +}\to \R^{s+t}$ by 
\[l(a)=(\log \vert \sigma_{1}(a)\vert,\dots ,\log \vert \sigma_{s}(a)\vert , 2\log \vert \sigma_{s+1}(a)\vert,\dots ,2\log \vert \sigma_{s+t}(a)\vert)
\]
for $a\in {\mathcal O}_{K}^{\ast\, +}$.
Then by Dirichlet's units theorem, $l({\mathcal O}_{K}^{\ast\, +})$ is a lattice in the vector space $L=\{x\in \R^{s+t}\vert \sum_{i=1}^{s+t} x_{i}=0\}$.
Consider the projection $p:L\to \R^{s}$ given by the first $s$ coordinate functions.
Then we have a  subgroup $U$ with the rank $s$ of ${\mathcal O}_{K}^{\ast\, +}$ such that $p(l(U))$ is a lattice in $\R^{s}$.
Write $l(U)=\Z v_{1}\oplus\dots\oplus \Z v_{s}$ for generators $v_{1},\dots v_{s}$ of $l(U)$.
For the standerd basis $e_{1},\dots ,e_{s+t}$ of $\R^{s+t}$, we have a regular real $s\times s$-matrix $(a_{ij})$ and $s\times t$ real constants  $b_{jk}$ such that
\[v_{i}=\sum_{j=1}^{s} a_{ij}(e_{j}+\sqrt{-1}\sum_{k=1}^{t}b_{jk}e_{s+k})
\]
for any $1\le i\le s$.
Consider the complex upper half plane $H=\{z\in \C: {\rm Im} z>0\}=\R\times \R_{>0}$.
We have the action of $U\ltimes{\mathcal O}_{K}$ on $H^{s}\times \C^{t}$
such that 
\begin{multline*}
(a,b)\cdot (x_{1}+\sqrt{-1}y_{1},\dots ,x_{s}+\sqrt{-1}y_{s}, z_{1},\dots ,z_{t})\\
=(\sigma_{1}(a)x_{1}+\sigma_{1}(b)+\sqrt{-1} \sigma_{1}(a)y_{1}, \dots ,\sigma_{s}(a)x_{s}+\sigma_{s}(b)+\sqrt{-1} \sigma_{s}(a)y_{s},\\
 \sigma_{s+1}(a)z_{1}+\sigma_{s+1}(b),\dots ,\sigma_{s+t}(a)z_{t}+\sigma_{s+t}(b)).
\end{multline*}
In \cite{OT} it is proved that the quotient $H^{s}\times \C^{t}/U\ltimes{\mathcal O}_{K}$ is compact.
We call this complex manifold a  Oeljeklaus-Toma (OT) manifold of type $(s,t)$.
\begin{thm}[\cite{OTH}]\label{OTHH}
The Fr\"olicher spectral sequence of every OT-manifold of  type $(s,t)$ degenerates at $E_{1}$-term.

\end{thm}

As in \cite{KV}, we present OT-manifolds as solvmanifolds considered  in the last section.
For $a \in U$ and $(x_{1},\dots ,x_{s})=p(l(a))\in p(l(U))$, since $l(U)$ is generated by the basis $v_{1},\dots,v_{s}$ as above, $l(a)$ is a linear combination of $e_{1}+\sum_{k=1}^{t}b_{1k}e_{s+k},\dots ,e_{s}+\sum_{k=1}^{t}b_{sk}e_{s+k}$ and hence we have 
\[l(a)=\sum_{i=1}^{s}x_{i}(e_{i}+\sum_{k=1}^{t}b_{ik}e_{s+k})=(x_{1},\dots ,x_{s}, \sum_{i=1}^{s}b_{i1}x_{i},\dots,\sum_{i=1}^{s}b_{it}x_{i}).
\]
By $2\log \vert \sigma_{s+k}(a)\vert=\sum_{i=1}^{s}b_{ik}x_{i}$, we can write 
\[\sigma_{s+k}(a)=e^{\frac{1}{2}\sum_{i=1}^{s}b_{ik}x_{i}+\sqrt{-1}\sum_{i=1}^{s}c_{ik}x_{i}}\]
for some $c_{ik}\in \R$.
We consider the Lie group $G=\R^{s}\ltimes_{\phi} (\R^{s}\times \C^{t})$ with
\[
\phi(x_{1},\dots ,x_{s})\\
={\rm diag}(e^{x_{1}},\dots ,e^{x_{s}},e^{\psi_{1}(x)},\dots ,e^{\psi_{t}(x)})\]
 where $\psi_{k}=\frac{1}{2}\sum_{i=1}^{s}b_{ik}x_{i}+\sum_{i=1}^{s}c_{ik}x_{i}$.
Then for $(x_{1},\dots ,x_{s})\in p(l(U))$, we have 
\[\phi(x_{1},\dots ,x_{s})(\sigma ({\mathcal O}_{K}))\subset \sigma({\mathcal O}_{K}).\]
Write $ p(l(U))=\Lambda$ and $ \sigma({\mathcal O}_{K})=\Delta$.
Then the OT-manifold $H^{s}\times \C^{t}/U\ltimes{\mathcal O}_{K}$ is identified with one of  the complex solvmanifolds  $(\Gamma\backslash G, J)$ as in the last section.
We consider the differential bigraded subalgebra  $B^{\ast,\ast}_{\Lambda}$ of Dolbeault complex of  OT-manifold $H^{s}\times \C^{t}/U\ltimes{\mathcal O}_{K}=\Gamma\backslash G$ with this solvmanifold presentation.

\begin{rem}\label{der}
We can compute the de Rham cohomology of OT-manifold $H^{s}\times \C^{t}/U\ltimes{\mathcal O}_{K}=\Gamma\backslash G$ by 
the finite dimensional graded subspace 
\[A^{\ast} _{\Lambda}= \left \langle dx_{I}\wedge dy_{J}\wedge dz_{K}\wedge d\bar{z}_{L}\left \vert\begin{array}{cc} I, J\subset [s], K,L\subset [t],  \\ \exp\left(\sum_{j\in J}x_{j}+\sum_{k\in K}\psi_{k}(x)+\sum_{l\in L}\bar{\psi}_{l}(x)\right)\vert _{x\in \Lambda}=1  \end{array} \right.   \right\rangle.
\] 
of $A^{\ast}(\Gamma\backslash G)$.
Since for $a\in U$ with $(x_{1},\dots ,x_{s})=p(l(a))\in p(l(U))$ we have $\sigma_{i}(a)=e^{x_{i}}$ for $1\le i\le s$,  $\sigma_{s+k}(a)=e^{\psi_{k}(x)}$ and $\sigma_{s+t+k}(a)=e^{\bar\psi_{k}(x)}$ for  $1\le k\le s$,
the condition $ \exp\left(\sum_{j\in J}x_{j}+\sum_{k\in K}\psi_{k}(x)+\sum_{l\in L}\bar{\psi}_{l}(x)\right)=1 $ is equivalent to 
the condition $\sigma_{JKL}(a)=1$
where we write $\sigma_{IJK}(a)=\Pi_{i}\sigma_{i}(a)\Pi_{j\in J}\sigma_{s+j}\Pi_{k\in K}\sigma_{s+t+k}(a)$.
Thus we obtain 
\[H^{\ast}(\Gamma\backslash G)_{\C}\cong A^{\ast} _{\Lambda}=
\bigwedge \langle dx_{1},\dots, dx_{s}\rangle \otimes \left \langle  dy_{J}\wedge dz_{K}\wedge d\bar{z}_{L}\left \vert\begin{array}{cc}  J\subset [s], K,L\subset [t],  \\ \sigma_{IJK}(a)=1, \forall a\in U  \end{array} \right.   \right\rangle
\] 
This is another  proof of  \cite[Theorem 3.1]{IO} by using the solvmanifold presentation of OT-manifolds.
\end{rem}

Then by Theorem \ref{OTHH} and Corollary \ref{iso}, we have:
\begin{thm}
The inclusion  $B^{\ast,\ast} _{\Lambda}\subset A^{\ast,\ast} (\Gamma\backslash G)$ induces an isomorphism $B^{\ast,\ast} _{\Lambda}\cong H^{\ast,\ast}_{\bar\partial} (\Gamma\backslash G)$.

\end{thm}
\begin{rem}
As noted in Remark \ref{sod}, this may be considered as a development  of computations of  Dolbeault cohomology of solvmanifolds.
\end{rem}

As noted in  Remark \ref{der}, we have:

\begin{cor}
\[H^{\ast,\ast}_{\bar\partial}(\Gamma\backslash G)\cong B^{\ast} _{\Lambda}=
\bigwedge \langle \bar\alpha_{1},\dots, \bar\alpha_{s}\rangle \otimes \left \langle  \exp(\Psi_{JKL}(x))\alpha_{J}\wedge \beta_{K}\wedge \bar\beta_{L}\left \vert\begin{array}{cc}  J\subset [s], K,L\subset [t],  \\ \sigma_{IJK}(a)=1, \forall a\in U  \end{array} \right.   \right\rangle
\] 
and hence
\[ \dim H^{p,q}_{\bar\partial}(\Gamma\backslash G)_{\C}= \sum_{p_{1}+p_{2}=p,\,\, q_{1}+q_{2}=q} \left (\begin{array}{cc} 
s\\
q_{1}
\end{array} \right)\rho_{p_{1}p_{2}q_{2}}
\]
where $\rho_{p_{1}p_{2}q_{2}}$ is the cardinal of the set 
\[\left\{ (J,K,L)\left\vert\begin{array}{cc} J\subset [s], K,L\subset [t], \\
 \vert J\vert=p_{1}, \vert K\vert=p_{2}, \vert L\vert=q_{2},\\
  \sigma_{IJK}(a)=1, \forall a\in U\end{array} \right. \right\}.\]
\end{cor}

By Corollary \ref{for}, we have:
\begin{thm}
 Every OT-manifold of  type $(s,t)$ is geometrically Dolbeault formal in particular Dolbeault formal.

\end{thm}

\begin{rem}
The formality of OT-manifolds  for de Rham complex is proved  in \cite{KF}.

\end{rem}

\section{Bott-Chern cohomology of OT-manifolds of type $(s,1)$}
Let $G$ be a simply connected solvable Lie group and $\g$ the Lie algebra of $G$. 
Suppose $G$ admits  a left-invariant complex structure $J$.
Then  we have the double complex structure $(\bigwedge \g^{\ast}_{\C}=\bigwedge^{\ast,\ast} \g^{\ast}_{\C}, d=\partial+\bar\partial)$ on the complex Lie algebra complex $\bigwedge \g^{\ast}_{\C}$.
We define the Lie algebra  Dolbeault cohomology $H^{\ast,\ast}_{\bar\partial}(\g)$.
As similar to the usual Dolbeault cohomology we have the  inequality $\sum_{p+q=k}\dim H^{p,q}_{\bar\partial}(\g)\ge \dim H^{k}(\g)$.
Suppose $G$ contains a cocompact discrete subgroup $\Gamma$.
then the inclusion $\bigwedge^{\ast,\ast} \g^{\ast}_{\C}\subset  A^{\ast,\ast} (\Gamma\backslash G)$ induces a canonical map $H^{\ast,\ast}_{\bar\partial}(\g)\to H^{\ast,\ast}_{\bar\partial}(\Gamma\backslash G)$.
This map is not an isomorphism in general but always an injection (cf \cite[Lemma 9]{CF}).
Hence we have the  inequalities 
\[  \sum_{p+q=k}\dim H^{p,q}_{\bar\partial}(\Gamma\backslash G) \ge\sum_{p+q=k}\dim H^{p,q}_{\bar\partial}(\g)\ge \dim H^{k}(\g).
\]
If the Fr\"olicher spectral sequence of $\Gamma\backslash G$ degenerates 
at $E_{1}$-term (i.e. $ \sum_{p+q=k}\dim H^{p,q}_{\bar\partial}(\Gamma\backslash G)=\dim H^{k}(\Gamma\backslash G))$ and an isomorphism $H^{\ast}(\g)\cong H^{\ast}(\Gamma\backslash G)$ holds, then the injection $H^{\ast,\ast}_{\bar\partial}(\g)\to H^{\ast,\ast}_{\bar\partial}(\Gamma\backslash G)$ is an isomorphism.

For  the double complex  $(\bigwedge \g^{\ast}_{\C}=\bigwedge^{\ast,\ast} \g^{\ast}_{\C}, d=\partial+\bar\partial)$, we define the Lie algebra Bott-Chern cohomology  $H^{\ast,\ast}_{BC}(\g)$.
By the result in \cite{AK}, if we have isomorphisms $H^{\ast}(\g)\cong H^{\ast}(\Gamma\backslash G)$ and $H^{\ast,\ast}_{\bar\partial}(\g)\cong H^{\ast,\ast}_{\bar\partial}(\Gamma\backslash G)$, then 
the  inclusion $\bigwedge^{\ast,\ast} \g^{\ast}_{\C}\subset  A^{\ast,\ast} (\Gamma\backslash G)$ induces an isomorphism 
 $H^{\ast,\ast}_{BC}(\g)\cong  H^{\ast,\ast}_{BC}(\Gamma\backslash G)$.

We consider an  OT-manifold $H^{s}\times \C^{t}/U\ltimes{\mathcal O}_{K}=\Gamma\backslash G$ as the last section.
We assume $t=1$.
Then, we have $\psi_{1}=\frac{1}{2}\sum_{i=1}^{s}x_{i}+\sqrt{-1}\sum_{i=1}^{s}c_{i1}x_{i}$.
It is known that an isomorphism $H^{\ast}(\g)\cong H^{\ast}(\Gamma\backslash G)$ holds see \cite[Section 10]{KM}.
Thus, by Theorem \ref{OTHH} and the above argument, the  inclusion $\bigwedge^{\ast,\ast} \g^{\ast}_{\C}\subset  A^{\ast,\ast} (\Gamma\backslash G)$ induces an isomorphism 
 $H^{\ast,\ast}_{BC}(\g)\cong  H^{\ast,\ast}_{BC}(\Gamma\backslash G)$.
 
 By the arguments in the last section, we have
 \[H^{\ast}(\Gamma\backslash G)_{\C}\cong A^{\ast} _{\Lambda}=
\bigwedge \langle dx_{1},\dots, dx_{s}\rangle \otimes \left   \langle 1,  dy_{[s]}\wedge dz_{1}\wedge d\bar{z}_{1}\  \right\rangle
\] 
and 
 \[H^{\ast,\ast}_{\bar\partial}(\Gamma\backslash G)\cong B^{\ast} _{\Lambda}=
\bigwedge \langle \bar\alpha_{1},\dots, \bar\alpha_{s}\rangle \otimes \left \langle 1, \alpha_{[s]}\wedge \beta_{1}\wedge \bar\beta_{1} \right\rangle
\] 
and hence
\[\dim H^{p,q}_{\bar\partial}(\Gamma\backslash G)=\left \{\begin{array}{cccc}
\left (\begin{array}{cc} 
s\\
q
\end{array} \right)  \qquad (p=0)\\
\left (\begin{array}{cc} 
s\\
q-1
\end{array} \right) \qquad (p=s+1)\\
1 \qquad\qquad (p=q=0)\\
0 \qquad\qquad (\rm otherwise)
    \end{array} \right. .
\]

We compute $H^{\ast,\ast}_{BC}(\Gamma\backslash G)$.
We should notice that we have the elliptic differential operator $\tilde\Delta_{BC}$ corresponding to the Bott-Chern cohomology of a Hermitian manifold (see \cite{S}).
It is known that a differential form $\alpha$ is BC-harmonic (i.e. $\tilde\Delta_{BC}\alpha=0$) if and only if $\partial\alpha=\bar\partial \alpha=(\partial\bar\partial)^{\ast}\alpha=0$ where $(\partial\bar\partial)^{\ast}$ is the formal adjoint of $\partial\bar\partial$.
As usual Hodge theory, we have a unique  BC-harmonic  representative of every cohomology class in the Bott-Chern cohomology.

\begin{prop}
\[H^{p,q}_{BC}(\Gamma\backslash G)=\left \{\begin{array}{cccc} \langle \alpha_{I}\wedge\bar\alpha_{[s]}\wedge \beta_{1}\wedge\bar\beta_{1}\vert I\subset [s], \vert I\vert=p-1 \rangle \qquad (q=s+1)
  \\ \langle \alpha_{[s]}\wedge\bar\alpha_{J}\wedge \beta_{1}\wedge\bar\beta_{1}\vert J\subset [s], \vert J\vert=q-1 \rangle \qquad (p=s+1)\\
  \langle \alpha_{1}\wedge \bar\alpha_{1},\dots ,  \alpha_{s}\wedge \bar\alpha_{s}\rangle \qquad (p=q=1)
  \\
 \langle   1 \rangle  \qquad\qquad (p=q=0)
  \\
  0 \qquad\qquad (\rm otherwise)
    \end{array} \right. 
\]
and hence
\[\dim H^{p,q}_{BC}(\Gamma\backslash G)=\left \{\begin{array}{cccc}
\left (\begin{array}{cc} 
s\\
p-1
\end{array} \right)  \qquad (q=s+1)\\
\left (\begin{array}{cc} 
s\\
q-1
\end{array} \right) \qquad (p=s+1)\\
s  \qquad  \qquad (p=q=1)\\
1 \qquad\qquad (p=q=0)\\
0 \qquad\qquad (\rm otherwise)
    \end{array} \right. .
\]
Moreover, each representative as above is BC-harmonic on the Hermitian manifold $(\Gamma\backslash G,h_{G})$ where $h_{G}$ is defined in Lemma \ref{injj}.
\end{prop}
\begin{proof}
We compute $ \rm ker\partial \cap ker \bar\partial\subset \bigwedge^{\ast,\ast} \g^{\ast}_{\C}$. 
$\bigwedge^{\ast,\ast} \g^{\ast}_{\C}$ is generated  four  type monomials
$\alpha_{I}\wedge\bar\alpha_{J}$, $\alpha_{I}\wedge\bar\alpha_{J}\wedge\beta_{1}$, $\alpha_{I}\wedge\bar\alpha_{J}\wedge\bar\beta_{1}$ and $\alpha_{I}\wedge\bar\alpha_{J}\wedge\beta_{1}\wedge\bar\beta_{1}$.
We have
\[\partial (\alpha_{I}\wedge\bar\alpha_{J})=-\frac{1}{2}\sum_{j\in J}\alpha_{j}\wedge\alpha_{I}\wedge \bar\alpha_{J}, \qquad \bar\partial (\alpha_{I}\wedge\bar\alpha_{J})=-\frac{1}{2}\sum_{i\in I}\bar\alpha_{i}\wedge\alpha_{I}\wedge \bar\alpha_{J},
\] 
\[\partial \beta_{1}=\frac{1}{4}\sum_{i\in [n]}\alpha_{i}\wedge \beta_{1}+\frac{1}{2}\sqrt{-1}\sum_{i\in [n]}c_{i1}\alpha_{i}\wedge \beta_{1},\qquad \partial\bar\beta_{1}=\frac{1}{4}\sum_{i\in [n]}\alpha_{i}\wedge \bar\beta_{1}-\frac{1}{2}\sqrt{-1}\sum_{i\in [n]}c_{i1}\alpha_{i}\wedge\bar \beta_{1}
\]
and 
\[\bar\partial \beta_{1}=\frac{1}{4}\sum_{i\in [n]}\bar\alpha_{i}\wedge \beta_{1}+\frac{1}{2}\sqrt{-1}\sum_{i\in [n]}c_{i1}\bar\alpha_{i}\wedge \beta_{1},\qquad \bar\partial\bar\beta_{1}=\frac{1}{4}\sum_{i\in [n]}\bar\alpha_{i}\wedge \bar\beta_{1}-\frac{1}{2}\sqrt{-1}\sum_{i\in [n]}c_{i1}\bar\alpha_{i}\wedge \bar\beta_{1}
\]
We can check that $ \rm ker\partial \cap ker \bar\partial\subset \bigwedge^{\ast,\ast} \g^{\ast}_{\C}$ is the direct sum of
\[\langle \alpha_{I}\wedge\bar\alpha_{I}\vert I\subset [n]\rangle
\]
and
\[\langle \alpha_{I}\wedge\bar\alpha_{J}\wedge \beta_{1}\wedge\bar\beta_{1}\vert I\subset [n], J\subset [n], \check{I}\subset J \rangle.
\]
Consider the conjugate Hodge star operator  $\bar\ast$ as in the proof of Lemma \ref{injj}.
Then we should know the subspace $ \rm ker\partial \cap ker \bar\partial\cap  ker (\partial\bar\partial \bar\ast)\subset \bigwedge^{\ast,\ast} \g^{\ast}_{\C}$.
We have
\[ (\partial\bar\partial \bar\ast)(\alpha_{I}\wedge \bar\alpha_{I})=\partial\bar\partial ( \alpha_{\check{I}}\wedge\bar\alpha_{\check{I}}\wedge \beta_{1}\wedge\bar\beta_{1})=\frac{1}{4}\sum_{i,i^{\prime}\in I, 
i\not=i^{\prime}} \alpha_{i}\wedge \bar\alpha_{i^{\prime}}\wedge\alpha_{\check{I}}\wedge\bar\alpha_{\check{I}}\wedge \beta_{1}\wedge\bar\beta_{1}
\]
and
\[ (\partial\bar\partial \bar\ast)(\alpha_{I}\wedge \bar\alpha_{J}\wedge \beta_{1}\wedge\bar\beta_{1})=\partial\bar\partial (\alpha_{\check{I}}\wedge \bar\alpha_{\check{J}})=\frac{1}{4}\sum_{i\in \check{I}, j\in \check{J} }\alpha_{j}\wedge\bar\alpha_{i}\wedge \alpha_{\check{I}}\wedge \bar\alpha_{\check{J}}.
\]
Thus it is sufficient to check term by term.
If $(\partial\bar\partial \bar\ast)(\alpha_{I}\wedge \bar\alpha_{I})=0$, then  $\vert I\vert =1$.
If $(\partial\bar\partial \bar\ast)(\alpha_{I}\wedge \bar\alpha_{J}\wedge \beta_{1}\wedge\bar\beta_{1})=0$ with $\check{I}\subset J $, then  $\check{I}\subset \check{J} $ or  $\check{J}\subset \check{I} $ and hence $I=[s]$ or $J=[s]$.
Hence the proposition follows
\end{proof}

For any $n$-dimensional complex manifold $M$, the inequality 
\[\sum_{p+q=k}\left(\dim H^{p,q}_{BC}(M)+\dim H^{n-p, n-q}_{BC}(M)\right)\ge 2\dim H^{k}(M)
\]
holds for every $k\le 2n$ (\cite[Theorem A]{AT}).
The  $\partial\bar\partial$-Lemma holds if and only if the equality 
\[\sum_{p+q=k}\left(\dim H^{p,q}_{BC}(M)+\dim H^{n-p, n-q}_{BC}(M)\right)= 2\dim H^{k}(M)
\]
holds for every $k\le 2n$ (\cite[Theorem B]{AT}).

By the above computation, on an OT-manifold $H^{s}\times \C^{1}/U\ltimes{\mathcal O}_{K}=\Gamma\backslash G$ of type $(s,1)$,  for $k\le s+1$ we have 
\[\dim H^{k}(\Gamma\backslash G)=\left \{\begin{array}{cccc}
\left (\begin{array}{cc} 
s\\
k
\end{array} \right)  \qquad (k\not= 0,s+1 )\\
0 \qquad \qquad (k=s+1)\\
s  \qquad \qquad  (k=1)\\
1 \qquad \qquad  (k=0)
    \end{array} \right., \]
  \[  \sum_{p+q=k}\dim H^{p,q}_{BC}(\Gamma\backslash G)=\left \{\begin{array}{cccc}
0 \qquad \qquad (k\not=2,0)\\
s  \qquad \qquad  (k=2)\\
1 \qquad \qquad  (k=0)
 \end{array} \right. \]
 and
 \[ \sum_{p+q=k}\dim H^{s+1-p, s+1-q}_{BC}(\Gamma\backslash G)=  \left \{\begin{array}{cccc}
2\left (\begin{array}{cc} 
s\\
s-k
\end{array} \right)  \qquad (k\not= 0, s+1 )\\
1 \qquad \qquad  (k=0)
\\
0 \qquad \qquad (k=s+1)\\
    \end{array} \right. .\]
    
 Hence we have:   
\begin{cor}
On an OT-manifold $H^{s}\times \C^{1}/U\ltimes{\mathcal O}_{K}=\Gamma\backslash G$ of type $(s,1)$, the  equality
\[\sum_{p+q=k}\left(\dim H^{p,q}_{BC}(\Gamma\backslash G)+\dim H^{s+1-p, s+1-q}_{BC}(\Gamma\backslash G)\right)= 2\dim H^{k}(\Gamma\backslash G)
\]

holds for every $k\not=2,  2s$ and does not hold for $k=2, 2s$.
\end{cor}

\begin{rem}
In case $s=1$ (Inoue surfaces of type $S^{0}$), these results are given by  the general theory of complex surfaces (see \cite{ADT}).

\end{rem}


\begin{thebibliography}{4000000}
\bibitem{AT}
D. Angella,  A. Tomassini,  On the $\partial\bar\partial$-Lemma and Bott-Chern cohomology. Invent. Math.{\bf 192} (2013), no. 1, 71--81.
\bibitem{ADT}
D. Angella, G. Dloussky, A. Tomassini, On Bott-Chern cohomology of compact complex surfaces. Ann. Mat. Pura Appl. (4) {\bf195} (2016), no. 1, 199--217. 
\bibitem{AK}
D. Angella, H. Kasuya,  Bott-Chern cohomology of solvmanifolds. Ann. Global Anal. Geom. {\bf 52} (2017), no. 4, 363--411.
\bibitem{APV}
D. Angella, M. Parton, V. Vuletescu, Rigidity of Oeljeklaus-Toma manifolds. arXiv:1610.04045 to appear in Ann. Inst. Fourier 
\bibitem{CF}
S. Console, A. Fino,  Dolbeault cohomology of compact nilmanifolds. Transform. Groups {\bf 6} (2001), no. 2, 111--124.
\bibitem{CFK}S. Console, A. Fino, H. Kasuya,  On de Rham and Dolbeault cohomology of solvmanifolds. Transform. Groups {\bf 21} (2016), no. 3, 653--680.
\bibitem{DGMS}P. Deligne, P. Griffiths, J. Morgan, and D. Sullivan, Real homotopy theory of Kahler manifolds. Invent. Math. {\bf 29} (1975), no. 3, 245--274.
\bibitem{IO} N. Istrati, A. Otiman,  De Rham and twisted cohomology of Oeljeklaus-Toma manifolds. Ann. Inst. Fourier (Grenoble) {\bf 69} (2019), no. 5, 2037--2066.

\bibitem{KV}H. Kasuya,  Vaisman metrics on solvmanifolds and Oeljeklaus-Toma manifolds. Bull. Lond. Math. Soc. {\bf 45} (2013), no. 1, 15--26.
\bibitem{KM} H. Kasuya, Minimal models, formality, and hard Lefschetz properties of solvmanifolds with local systems. J. Differential Geom. {\bf93} (2013), no. 2, 269--297.
\bibitem{KF}H. Kasuya, Formality and hard Lefschetz property of aspherical manifolds. Osaka J. Math. {\bf 50} (2013), no. 2, 439--455.
\bibitem{KR} H. Kasuya,
de Rham and Dolbeault cohomology of solvmanifolds with local systems. Math. Res. Lett.{\bf 21} (2014), no. 4, 781--805.
\bibitem{KN}H. Kasuya,
An extention of Nomizu's Theorem—a user's guide.Complex Manifolds {\bf 3} (2016), no. 1, 231--238. 
\bibitem{NT}
J. Neisendorfer,  L. Taylor, Dolbeault homotopy theory. Trans. Amer. Math. Soc. {\bf 245} (1978), 183--210. 
\bibitem{OT}K. Oeljeklaus, M. Toma,  Non-K\"ahler compact complex manifolds associated to number fields. Ann. Inst. Fourier (Grenoble) {\bf 55} (2005), no. 1, 161--171.
\bibitem{OTH}
A. Otiman, M. Toma, Hodge decomposition for Cousin groups and Oeljeklaus-Toma manifolds, arXiv: 1811.02541, to appear in Annali della Scuola Normale di Pisa
\bibitem{R} M. S. Raghunathan, Discrete subgroups of Lie Groups, Springer-Verlag, New York, 1972.
\bibitem{S}
M. Schweitzer, Autour de la cohomologie de Bott-Chern, arXiv:0709.3528.

\end{thebibliography}
\end{document}